\documentclass{amsart}

 \usepackage{amsmath}
\usepackage{amsfonts}
\usepackage{bbm}
\usepackage{amssymb}
\usepackage[dvips]{graphicx}
\usepackage{color}
\usepackage{float} 
\usepackage{url}
\usepackage{fullpage}
\usepackage{epsfig} 

\newtheorem{thm}{\bf{Theorem}}[section]
\newtheorem{lemma}[thm]{\bf{Lemma}}

\newtheorem{df}[thm]{\bf{Definition}}

\newtheorem{prop}[thm]{\bf{Proposition}}

\newtheorem{ex}[thm]{\bf{Example}}
\newtheorem{rem}[thm]{\bf{Remark}}

\begin{document}

\title[]{On the strong concavity of the dual function of an optimization problem}

\date{}

\maketitle

\begin{center}
\begin{tabular}{c}
Vincent Guigues\\
School of Applied Mathematics, FGV\\
Rio de Janeiro, Brazil\\ 
(\tt{vincent.guigues@fgv.br})
\end{tabular}
\end{center}

\begin{abstract} 
We provide three new proofs of the strong concavity of the dual function
of some convex optimization problems. 
For problems with nonlinear constraints, we show 
that the assumption of strong convexity of the objective cannot be weakened to convexity and
that the assumption
that the gradients of all constraints at the optimal solution
are linearly independent cannot be further weakened.
Finally, we illustrate our results
with several examples.\\
\end{abstract}



\section{Introduction}

Consider the optimization problem 
\begin{equation} \label{pbinit0}
\left\{
\begin{array}{l}
\inf f(x)\\
A x \leq b
\end{array}
\right.
\end{equation}
where $f:\mathbb{R}^n \rightarrow \mathbb{R}$,
$b \in \mathbb{R}^q$, and $A$ is a $q \times n$ real matrix.

We make the following assumptions:
\begin{itemize}
\item[(H1)] $f$ is convex, differentiable, and $\nabla f$ is Lipschitz continuous on $\mathbb{R}^n$: there is $L(f)>0$ such that for every $x,y \in \mathbb{R}^n$ we have
$$
\|\nabla f(y)-\nabla f(x)\|_2 \leq L(f) \|y-x\|_2.
$$
\item[(H2)] The rows of the matrix $A$ are
linearly independent.
\end{itemize}
Let $\theta$ be the dual function of \eqref{pbinit0} given by
\begin{equation}\label{dualfunctionfirst}
\theta( \lambda ) = \displaystyle  \inf_{x \in \mathbb{R}^n} 
\mathcal{L}(x,\lambda):=f(x) + \lambda^T ( Ax  - b ),
\end{equation}
for $\lambda \in \mathbb{R}^q$.  The function $\theta$ is concave \cite{claude, Rockafellar}
and it was shown in 
\cite{vguiguesMP2020} that under Assumptions (H1), (H2), it 
is strongly concave on $\mathbb{R}^q$ with constant of strong concavity
$\frac{\lambda_{\min}(A A^T)}{L(f)}$. This property was also shown
in \cite{YuNeely2015} with stronger 
assumptions, namely assuming
$f$ strongly convex and twice continuously differentiable.

For more general convex problems
of the form 
\begin{equation}\label{defpblinnonlinstrongconc}
\inf_{x \in \mathbb{R}^n} \{ f(x) :  Ax \leq b, g_i(x) \leq 0, i=1,\ldots,p\},
\end{equation}
where $f,g_i:\mathbb{R}^n \rightarrow \mathbb{R}$ are convex,
$b \in \mathbb{R}^q$, and $A$ is a $q \times n$ real matrix, the local
strong concavity  of the dual function obtained dualizing all constraints
was shown in \cite{vguiguesMP2020} in a neighborhood of an optimal dual solution. Such a property was also shown in \cite{YuNeely2015} with stronger
assumptions, namely assuming functions $f$ and $g$ twice
continuously differentiable whereas in \cite{vguiguesMP2020} it was only
assumed that functions $f, g_i$ have Lipschitz continuous gradients (in both
proofs, strong convexity of $f$ was also used).

The strong concavity of 
the dual function can be used
to design efficient solution methods on the dual problem, for instance
the Drift-Plus-Penalty Algorithm described in \cite{YuNeely2015}.
It was also used in \cite{vguiguesMP2020}
to compute inexact cuts for the recourse function
 of a two-stage convex stochastic program.
 These inexact cuts are useful to design 
Inexact Stochastic Mirror Descent (ISMD)
Method, introduced in \cite{vguiguesMP2020},
which is an inexact variant of
Stochastic Mirror Descent (SMD, see \cite{nemlansh09}).

In this paper, we provide three new proofs for the strong concavity of $\theta$ 
given 
by \eqref{dualfunctionfirst}. The first one uses the  assumptions from \cite{vguiguesMP2020}, the second one applies when $f$ is coercive while the 
third ones applies when $f$ is twice continuously differentiable.
We also show that for problems of the form \eqref{defpblinnonlinstrongconc},
the assumption
that the gradients of all constraints at the optimal solution
are linearly independent cannot be further weakened
and that the assumption of strong convexity of the objective cannot
be weakened to convexity. Finally, several examples are given.
\section{Preliminaries} 

Given a convex function $f:\mathbb{R}^n\to \mathbb{R}\cup\{-\infty,+\infty\}$, as usual, 
	$\mathrm{dom}\,f:=\{x\in \mathbb{R}^n: f(x)< +\infty\}$ is its {\em (effective) domain} and for any $x\in \mathbb{R}^n$ with $|f(x)|<+\infty$ the subdifferential $\partial{f}(x)$ is defined by 
	$$
	  \partial{f}(x):=\{s\in \mathbb{R}^n: \langle s,y-x\rangle \leq f(y) -f(x),\;\forall y
		\in \mathbb{R}^n \},  
	$$
	and $\partial{f}(x):=\emptyset$ if $|f(x)|=+\infty$. The function $f$ is said to be {\em proper} if $\mathrm{dom}\,f\not=\emptyset$ and $f$ does not take the value $-\infty$. The Legendre-Fenchel conjugate 
	$f^*$ on $\mathbb{R}^n$ is defined by 
	$$
	   f^*(y):=\sup_{x\in \mathbb{R}^n}\{\langle y,x\rangle -f(x)\} \quad\text{for all}\; 
		 y\in \mathbb{R}^n.
	$$
	Similarly, given a concave function $g:\mathbb{R}^n\to \mathbb{R}\cup\{-\infty,+\infty\}$ (i.e., 
	$-g$ is convex), its {\em (effective) domain} is the set 
	$\{x\in \mathbb{R}^n: g(x) > -\infty\}$. We refer to \cite{claude,Rockafellar} for the above concepts.

In what follows, $X \subset \mathbb{R}^n$ is a nonempty convex set, and its relative interior will be denoted by $\mathrm{ri}\,X$.
 
\begin{df}[Strongly convex functions]\label{defsconv}  A function $f:\mathbb{R}^n \rightarrow \mathbb{R}\cup \{+\infty\}$ is strongly convex 
on $X$ with constant of strong convexity $\alpha > 0$  with respect to a norm $\|\cdot\|$ if
for any $x, y \in X\cap \mathrm{dom}\,f$ we have 
\begin{equation}\label{scfdomain}
f( t x  + (1-t)y ) \leq t f(x) + (1-t)f(y)  - \frac{\alpha t(1-t)}{2} \|y-x\|^2, 
\end{equation}
for all $0 \leq t \leq 1$.
\end{df}
We can show that for
lower semicontinuous convex functions, strong
convexity on the relative interior
of the domain implies strong convexity
everywhere. More precisely, we have
the following:
\begin{lemma}\label{simplelemmaconv} Let
$f:\mathbb{R}^n \rightarrow \mathbb{R}\cup \{+\infty\}$ 
be a proper lower semicontinuous convex function which is 
strongly convex on 
the relative interior of its domain
with constant of strong convexity $\alpha>0$
with respect to norm $\|\cdot\|$, i.e., $f$
satisfies \eqref{scfdomain}
for $x, y \in \emph{ri(dom}(f))$ and
$0 \leq t \leq 1$. Then
$f$ is strongly convex on $\mathbb{R}^n$ 
with the same constant of strong convexity $\alpha>0$
with respect to norm $\|\cdot\|$
\end{lemma}
\begin{proof} 
Take any $x, y \in \mbox{dom}(f)$ and
$0<t<1$. 
We want to show \eqref{scfdomain}.
Since
$\mbox{ri}(\mbox{dom}(f))$ is nonempty, we can
take an arbitrary point 
$x_0 \in  \mbox{ri}(\mbox{dom}(f))$.
Observe that since $f$ is lower semicontinuous
and convex  with $x,y \in \mathrm{dom}\,f$, 
using Proposition 1.2.5 in \cite{claude} we have
\begin{equation}\label{firstconvlemm}
\begin{array}{l}
f(x)=\displaystyle{\varliminf_{u \rightarrow x}} f(u)
=\displaystyle{\lim_{\theta \rightarrow 0, \theta>0}} f(x+\theta(x_0-x)),\\
f(y)=\displaystyle{\varliminf_{u \rightarrow y}} f(u)
=\lim_{\theta \rightarrow 0, \theta>0} f(y+\theta(x_0-y)),
\end{array}
\end{equation}
and by lower semicontinuity of $f$ at  $tx+(1-t)y$ we also have
\begin{equation}
\begin{array}{lcl}
f(tx + (1-t)y)&\leq &  \displaystyle{\lim_{\theta \rightarrow 0, \theta>0}} f(tx+(1-t)y+\theta(x_0-tx-(1-t)y))\\  
&=& \displaystyle{\lim_{\theta \rightarrow 0, \theta>0}} f(t (x+\theta(x_0-x)) + (1-t)(y+\theta(x_0-y)))\\
& \leq & \displaystyle{\lim_{\theta \rightarrow 0, \theta>0}} t f(x+\theta(x_0-x))+ (1-t)f(y+\theta(x_0-y))-\frac{\alpha t(1-t)}{2}\|(1-\theta)(x-y)\|^2\\
& \stackrel{\eqref{firstconvlemm}}{=} & tf(x) +(1-t)f(y)-\frac{\alpha t (1-t)}{2}\|y-x\|^2,
\end{array}
\end{equation}
where in the inequality above we have used the fact that 
$x+\theta(x_0-x), y+\theta(x_0-y) \in 
\mbox{ri(dom}(f))$
and that $f$ is strongly convex on $\mbox{ri(dom}(f))$.$\hfill$
\end{proof}

It is well known 
(see for instance Proposition 6.1.2 in \cite{claude}) that if $f$ is strongly convex
with constant $\alpha$ with respect to norm
$\|\cdot\|$ 
and subdifferentiable on $X$ (i.e.,
the subdifferential $\partial f(x)$
of $f$ at $x$ is nonempty for every
$x \in X$) then for all $x,y \in X$, we have
$$
f(y) \geq f(x) + s^T (y-x)   + \frac{\alpha}{2}\|y-x\|^2, \;\forall s \in \partial f(x).
$$
Therefore, using the notation (here and in
what follows) $\langle x , y \rangle =x^T y$
for $x,y \in \mathbb{R}^n$,
since for a function $f$ satisfying (H1) we must have
$$
f(y) \leq f(x) + \langle \nabla f(x), y-x \rangle 
+ \frac{L(f)}{2}\|y-x\|_2^2,
$$
for all $x,y \in \mathbb{R}^n$, if $f$ satisfies
(H1) and is strongly convex on $\mathbb{R}^n$
with constant of strong convexity $\alpha$ with respect
to norm $\|\cdot\|_2$ then we must have $\alpha \leq L(f)$.

We also recall that  a convex function $f: \mathbb{R}^n \rightarrow \mathbb{R}\cup \{+\infty\}$  is strongly convex on 
$\mbox{ri}(\mbox{dom}(f))$ with 
constant of strong convexity $\alpha > 0$  with respect to norm 
$\|\cdot\|$ if and only if
for every $x, y \in \mbox{ri}(\mbox{dom}(f))$ we have
\begin{equation}\label{charactscaux}
\langle \sigma -  s , y-x \rangle  \geq \alpha \|y-x\|^2,
\;\mbox{ for all }\sigma \in \partial f(y), s \in \partial f(x).
\end{equation}
For a proof of the above characterization \eqref{charactscaux},
see the proof of Theorem 6.1.2 in  \cite{claude}.

Finally, if $f: \mathbb{R}^n \rightarrow \mathbb{R}$ is twice differentiable
then 
$f$ is strongly convex on 
$\mathbb{R}^n$ with constant of strong convexity $\alpha > 0$ 
with respect to norm $\|\cdot\|_2$
if and only if for every $x \in \mathbb{R}^n$, we have
$\nabla^2 f(x) \succeq \alpha I_n$(see for instance Proposition 1 in  \cite{bauschke9}).

\begin{df}[Strongly concave functions] A function $f: \mathbb{R}^n \rightarrow \mathbb{R}\cup \{-\infty\}$
is strongly concave on $X$ with constant of strong concavity $\alpha > 0$ 
with respect to norm $\|\cdot\|$
if and only if $-f$ is strongly convex
on $X$ with constant of strong convexity $\alpha > 0$ 
with respect to norm $\|\cdot\|$.
\end{df}

We recall two well known results of convex analysis that will be used
in the sequel. 
\begin{prop}\label{conjstconvex} 
Let $f:\mathbb{R}^n \rightarrow \mathbb{R}\cup \{+\infty\}$ be a convex and lower semicontinuous function. Then
$f^*$ is strongly convex on $\mathbb{R}^n$ with constant of strong convexity $\alpha>0$ for norm $\|\cdot\|_2$ if and only if $f$ is differentiable and $\nabla f$ is Lipschitz continuous  on $\mathbb{R}^n$ with
constant $1/\alpha$ for norm $\|\cdot\|_2$.
\end{prop}
For a proof of the previous proposition, see 
the proof of Proposition 12.60 in \cite{Rockafellar}. The following result is known as Baillon-Haddad Theorem 
that we specialize
to functions $f:\mathbb{R}^n \rightarrow \mathbb{R}$:
\begin{thm}\label{theorBailHadd}\cite{baillonhaddad} 
Let 
$f:\mathbb{R}^n \rightarrow \mathbb{R}$
be convex, differentiable and satisfying Assumption (H1)
for some $0< L(f)<\infty$. Then $\nabla f$ 
is $1/L(f)$-co-coecive, meaning that for all
$x,y \in \mathbb{R}^n$ we have 
 \begin{equation}\label{cocoercive}
\langle y - x , \nabla f(y ) - \nabla f(x ) \rangle  \geq \frac{1}{L(f)}\|\nabla f( y )  -  \nabla f ( x)\|_2^2.
\end{equation}
\end{thm} 

\begin{rem}
Theorem \ref{theorBailHadd} 
follows from Proposition \ref{conjstconvex} and
property \eqref{charactscaux}.
\end{rem}

\section{Problems with linear constraints}

\subsection{Proofs of the strong concavity of the dual function}

In this section we provide several proofs of the following proposition, first proved in \cite{vguiguesMP2020}.
\begin{prop}\label{dualfunctionpbinit0}
Let Assumptions (H1) and (H2) hold.  Then the dual function $\theta$
given by \eqref{dualfunctionfirst}
 is strongly concave on $\mathbb{R}^q$
 with constant of strong concavity $ \frac{\lambda_{\min}( A A^T )}{L(f)}$ with respect to norm $\|\cdot\|_2$.
\end{prop}

We first recall the proof of Proposition \ref{dualfunctionpbinit0} given in 
\cite{vguiguesMP2020}.\\
\par \textbf{Proof of Proposition \ref{dualfunctionpbinit0} from \cite{vguiguesMP2020}.} 
The dual function of \eqref{pbinit0} given by \eqref{dualfunctionfirst}
can be written
\begin{equation}\label{thetaintermsoffstar}
\begin{array}{lll}
\theta( \lambda )& =&\displaystyle  \inf_{x \in \mathbb{R}^n} \{f(x) + \lambda^T  ( A x - b ) \} = 
-\lambda^T b - \sup_{x \in \mathbb{R}^n} \{ -x^T A^T \lambda  -f(x) \} \\
& = & - \lambda^T b - f^*( -A^T \lambda ) \mbox{ by definition of }f^*.
\end{array}
\end{equation}
From Assumption (H1) and Proposition \ref{conjstconvex}, 
$-f^*$ is strongly concave with constant of strong concavity 
$1/L(f)$. Assumption (H2) implies 
$\mbox{Ker}(AA^T)=\{0\}$ which,
together with the strong concavity of 
$-f^*$, easily implies that 
$\lambda \rightarrow -f^*(-A^T\lambda)$ is strongly concave
with constant of strong concavity $\frac{\lambda_{\min}(A A^T)}{L(f)}$
(see Proposition 2.5 in \cite{vguiguesMP2020} for details)
and therefore so is $\theta$ which is the sum of a linear function
and of $-f^*(-A^T\lambda)$.$\hfill \square$\\

\par Our second proof is based on Theorem \ref{theorBailHadd}.\\

\par {\textbf{Second proof of Proposition \ref{dualfunctionpbinit0}.}} 
If dom($-\theta$) is empty there
is nothing to show. Let us now
assume that dom($-\theta$) is
nonempty. 
Since $-\theta$ is convex,
the subdifferential
$\partial (-\theta)(\lambda)$
is nonempty for every $\lambda$
in the relative
interior of dom($-\theta$). 
We first show that for 
$x,y \in \mbox{ri}(\mbox{dom}(-\theta))$
relation \eqref{charactscaux}
holds for $f=-\theta$, $\alpha=\frac{\lambda_{\min}(A A^T )}{L(f)}$
and $\|\cdot\|=\|\cdot\|_2$. Let us take  
$\lambda_1, \lambda_2 \in 
\mbox{ri}(\mbox{dom}(-\theta))$.
Then
$\partial (-\theta)(\lambda_1)$
and 
$\partial (-\theta)(\lambda_2)$
are nonempty.  Using for instance Lemma 2.1 in \cite{guiguessiopt2016} or 
Corollary 4.5.3 p.273
  in \cite{claude}, we have 
\begin{equation}\label{gradtheta1}
-\partial(- \theta)(\lambda) = \{Ax-b : x \in S(\lambda)\} 
\end{equation}
where $S(\lambda)$ is the set of optimal solutions of \eqref{dualfunctionfirst}.
The latter equality yields that 
$S(\lambda_1)$ and $S(\lambda_2)$ are nonempty. Next, for
any $x_1 \in S(\lambda_1)$
and $x_2 \in S(\lambda_2)$,
by first order optimality conditions,
we have
\begin{equation}\label{eqxlamb} 
\nabla f( x_i  ) + A^T \lambda_i = 0,\;i=1,2.
\end{equation}
Now for any  
$s_2  \in {\partial} (-\theta)(\lambda_2)$
and $s_1  \in {\partial}(- \theta)(\lambda_1)$
we can find $x_1 \in S(\lambda_1)$
and $x_2 \in S(\lambda_2)$
such that 
$s_i =-(Ax_i-b)$,
$i=1,2$, which implies
\begin{equation}\label{secondsctheta2}
\begin{array}{lcl}
 \left \langle s_2  - s_1  ,  
\lambda_2 - \lambda_1  \right \rangle 
& = &
\displaystyle  - \langle A (x_2 - x_1) , \lambda_2 - \lambda_1 \rangle \mbox{ using }\eqref{gradtheta1}\\
& = &\displaystyle  - \langle x_2 - x_1 , A^T( \lambda_2 - \lambda_1 )  \rangle  \\
&\displaystyle  = & \langle x_2 - x_1 , \nabla f( x_2 ) -  \nabla f(x_1 ) \rangle 
\mbox{ using }\eqref{eqxlamb}\\
& \geq & \displaystyle   (1/L(f)) \|\nabla f( x_2) -  \nabla f(x_1) \|_2^2\mbox{ using }\eqref{cocoercive}\\
& = & \displaystyle (1/L(f)) \| A^T ( \lambda_2 - \lambda_1 )  \|_2^2 \mbox{ using }\eqref{eqxlamb}\\
& \geq &\displaystyle  \frac{\lambda_{\min}(A A^T )}{L(f)} \| \lambda_2 - \lambda_1   \|_2^2.
\end{array}
\end{equation}
Recalling characterization 
\eqref{charactscaux} of strong convexity, we have shown that $-\theta$ is strongly convex with constant
of strong convexity $\frac{\lambda_{\min}(A A^T )}{L(f)}$ with respect to 
norm $\|\cdot\|_2$ on 
$\mbox{ri}(\mbox{dom}(-\theta))$.
Since $-\theta$ is convex
and lower semicontinuous on $\mathbb{R}^q$,
we can apply Lemma \ref{simplelemmaconv}
with $X=\mathbb{R}^q$
to obtain the strong convexity
of $-\theta$ (or equivalently the
strong concavity of $\theta$)
on $\mathbb{R}^q$ with constant
of strong convexity 
$\frac{\lambda_{\min}(A A^T )}{L(f)}$ with respect to 
norm $\|\cdot\|_2$.
$\hfill \square$\\

Our next proof of the strong concavity of the dual function applies when the objective
function $f$ is coercive. 
It is based on properties
of the value function.
\begin{prop}
Let Assumptions (H1) and (H2) hold
and assume that $f$ is 
coercive in the sense that $f(x)\to +\infty$ as 
$\|x\| \to +\infty$.  Then the dual function $\theta$
given by \eqref{dualfunctionfirst}
 is strongly concave on $\mathbb{R}_{+}^q$
 with constant of strong concavity 
 $\frac{\lambda_{\min}(A A^T )}{L(f)}$ with respect to norm $\|\cdot\|_2$.
\end{prop}

\begin{proof}
Let $v$ be the value function 
given by
\begin{equation}\label{valuedefv}
v(c)=\left\{
\begin{array}{l}
\inf f(x)\\
A x - b + c \leq  0,
\end{array}
\right.
\end{equation}
for $c \in \mathbb{R}^q$.

We first show that (i) $v$ is differentiable
and (ii) $\nabla v$
is Lipschitz continuous with 
Lipschitz constant
$\frac{L(f)}{\lambda_{\min}(A A^T )}$.

Let us show (i). By Assumption (H2), $A$ is
surjective meaning that for any
$y\leq b-c$ we can find $x$ satisfying
$Ax=y\leq b-c$ which is feasible
for \eqref{valuedefv} (due to
(H2), problem \eqref{valuedefv}
is in fact strictly feasible
for every $c$).
Therefore, for any $c$,
problem \eqref{valuedefv}
is convex, with 
polyhedral nonempty feasible set, and
continuous
coercive objective function, implying
that it has a finite optimal value
$v(c)$ and optimal solutions.
Also since $v$ is convex and 
finite for any $c$, it is a continuous convex function, hence its subdifferential
is nonempty at any $c$ and
is given by the set of optimal dual solutions of the dual problem
\begin{equation}\label{duallac}
\sup_{\lambda \geq 0} \theta_c(\lambda),
\end{equation}
where
$$
\theta_c(\lambda)=\inf_{x \in \mathbb{R}^n} f(x) + \lambda^T(Ax-b+c).
$$
Take $c \in \mathbb{R}^q$.
To show that $v$ is differentiable 
at $c$,
if suffices
to show that $\partial v(c)$ 
is a singleton. Take $\lambda_1, \lambda_2 \in \partial v(c)$, which, as we recall,
are optimal solutions to dual
problem \eqref{duallac}.
Let $x(c)$ be an optimal solution of  
\eqref{valuedefv} (recall
that \eqref{valuedefv} has optimal primal
and dual solutions). By the optimality conditions, we have
$\nabla f(x(c))+A^T \lambda_1=0$ and $\nabla f(x(c))+A^T \lambda_2=0$.
Therefore $\lambda_1-\lambda_2 \in \mbox{Ker}(A^T)$ 
and by Assumption (H2) we have Ker$(A^T)=\{0\}$ which implies $\lambda_1=\lambda_2$, i.e., 
$\partial v(c)$ is a singleton and (i) is shown. Therefore, for each $c\in \mathbb{R}^q$ there is a unique multiplier $\lambda(c)=\nabla v(c)$.

Let us now show (ii).
We will still denote by $x(c)$ an optimal solution of  
\eqref{valuedefv}. Take $c_1, c_2 \in \mathbb{R}^q$. By the optimality
conditions, we get
\begin{equation}\label{optcond}
\nabla f(x(c_i)) + A^T \lambda(c_i) = 0,i=1,2,
\end{equation}
and by complementary slackness
\begin{equation}\label{compslack}
\left \langle \lambda( c_i) , A x(c_i)-b+c_i \right \rangle  = 0,i=1,2.
\end{equation}
Therefore
\begin{equation}\label{secondsctheta22}
\begin{array}{lcl}
\left \langle A(x( c_2)  - x(c_1) )  ,   \lambda(c_1) - \lambda(c_2)  \right \rangle & = & \displaystyle \left \langle x( c_2)  - x(c_1)   , A^T (\lambda(c_1) - \lambda(c_2) ) \right \rangle \\
& \stackrel{\eqref{optcond}}{=} &\displaystyle  \left \langle x( c_2)  - x(c_1), \nabla f( x( c_2 )) -  \nabla f(x( c_1 )) \right \rangle  \\
& \stackrel{\eqref{cocoercive}}{\geq} &\displaystyle  (1/L(f)) \|\nabla f( x( c_2 )) -  \nabla f(x( c_1 )) \|_2^2\\
& \stackrel{\eqref{optcond}}{=} &\displaystyle  (1/L(f)) \| A^T ( \lambda(c_2) - \lambda(c_1) )  \|_2^2 \\
& \geq & \displaystyle \frac{\lambda_{\min}(A A^T )}{L(f)} \| \lambda(c_2) - \lambda(c_1)   \|_2^2.
\end{array}
\end{equation}
Observe that in
the first
inequality above, \eqref{cocoercive}
can be used because (H1) is satisfied. Next since $\lambda(c_1), \lambda(c_2) \geq 0$, we have 
\begin{equation}\label{feasib}
\left \langle \lambda(c_1) ,  A x(c_2)-b+c_2 \right \rangle \leq 0,\;
\left \langle \lambda(c_2) ,   A x(c_1)-b+c_1 \right \rangle \leq 0.
\end{equation}
It follows that 
\begin{equation}\label{vlipschitizfinal}
\begin{array}{l}
\left \langle A(x( c_2)  - x(c_1) )  ,   \lambda(c_1) - \lambda(c_2)  \right \rangle  = 
\left \langle Ax( c_2)-b  - ( A x(c_1)-b )  ,   \lambda(c_1) - \lambda(c_2)  \right \rangle \\
  \stackrel{\eqref{compslack}}{=}  
\left \langle \lambda(c_1), c_1 \right \rangle + \left \langle \lambda(c_2), c_2 \right \rangle + \left \langle \lambda(c_1), A x(c_2) -b \right \rangle  + \left \langle \lambda(c_2), A x(c_1) -b \right \rangle \\
 \stackrel{\eqref{feasib}}{\leq}  
\left \langle \lambda(c_2) - \lambda(c_1), c_2 - c_1 \right \rangle \leq  \| \lambda(c_2) - \lambda(c_1)\|_2 \|c_2 - c_1\|_2,
\end{array}
\end{equation}
where the last inequality is due to the 
Cauchy-Schwartz inequality.

Combining \eqref{secondsctheta22} and \eqref{vlipschitizfinal} we get 
$$
\|\nabla v(c_2) - \nabla v(c_1)\|_2=
\|\lambda(c_2) - \lambda(c_1)\|_2 \leq \frac{L(f)}{\lambda_{\min}(A A^T )} \|c_2 - c_1\|_2.
$$
Therefore, we have shown that
$v$ is differentiable and has
Lipschitz continuous gradient
with Lipschitz constant
$\frac{L(f)}{\lambda_{\min}(A A^T )}$.
Recalling that $v$ is convex,
using Proposition \ref{conjstconvex} we
deduce that
$v^*$ is strongly convex on $\mathbb{R}^q$
with constant of strong 
convexity $\frac{\lambda_{\min}(A A^T )}{L(f)}$.
Since $v^*=-\theta$ on $\mathbb{R}_+^q$, this shows
the strong concavity of $\theta$
on $\mathbb{R}_+^q$
with the constant of
strong concavity $\frac{\lambda_{\min}(A A^T )}{L(f)}$.\hfill
\end{proof}

\par {\textbf{Proofs of the strong concavity of $\theta$ for strongly convex $f$.}}
The proof of the strong concavity of $\theta$ when, additionally to
(H1) and (H2), the function $f$ is strongly convex is known.
It can be seen as a special case of Theorem 10 in \cite{YuNeely2015}.
For completeness, we provide below a simple proof of this result
and provide a new proof when $f$ is twice continuously differentiable.

\begin{prop}
Let Assumptions (H1) and (H2) hold.
Assume that $f$ is strongly convex on $\mathbb{R}^n$
with constant of strong convexity $\alpha$
with respect to
$\|\cdot\|_2$.
  Then the dual function $\theta$
given by \eqref{dualfunctionfirst}
 is strongly concave on $\mathbb{R}^q$
 with constant of strong concavity $\frac{\alpha \lambda_{\min}(A A^T )}{L(f)^2}$ with respect to norm $\|\cdot\|_2$.
\end{prop}
\begin{proof} For any $\lambda \in \mathbb{R}^q$, due to the strong convexity of $f$, optimization problem
\eqref{dualfunctionfirst} has a unique optimal solution denoted by 
$x(\lambda)$. By \eqref{gradtheta1}, $\theta$ is differentiable
with $\nabla \theta(\lambda)=Ax(\lambda)-b$.
It follows that
\begin{equation}\label{secondsctheta4}
\begin{array}{lcl}
- \left \langle \nabla \theta( \lambda_2 ) - \nabla \theta( \lambda_1 ) ,  
\lambda_2 - \lambda_1  \right \rangle 
& = &
\displaystyle  - \langle A (x(\lambda_2) - x(\lambda_1)) , \lambda_2 - \lambda_1 \rangle,\\
& = &\displaystyle  - \langle x(\lambda_2) - x(\lambda_1) , A^T( \lambda_2 - \lambda_1 )  \rangle  \\
&\displaystyle  \stackrel{\eqref{eqxlamb}}{=} & \langle x(\lambda_2) - x(\lambda_1) , \nabla f( x( \lambda_2 )) -  \nabla f(x( \lambda_1 )) \rangle  \\
& \geq & \alpha \| x( \lambda_2 ) -  x( \lambda_1 ) \|_2^2 \\
& \stackrel{(H1)}{\geq} & \displaystyle (\alpha/L(f)^2) \|\nabla f(x( \lambda_2 ))- \nabla f(x( \lambda_1 ))   \|_2^2 \\
& = & \displaystyle (\alpha/L(f)^2) \| A^T(\lambda_2 - \lambda_1)\|_2^2 \\
& \geq &\displaystyle  \frac{\alpha \lambda_{\min}(A A^T )}{L(f)^2} \| \lambda_2 - \lambda_1   \|_2^2.
\end{array}
\end{equation}
In \eqref{secondsctheta4}, the first equality
comes from $\nabla \theta(\lambda)=Ax(\lambda)-b$,
the first inequality comes from the strong convexity of 
$f$, while the last equality comes 
from the optimality conditions.
This achieves the proof. \hfill
\end{proof}

Since we must have $\alpha \leq L(f)$, we get a smaller constant
of strong concavity than in the previous case where $f$ was not necessarily strongly
convex. We now provide a new proof when 
$f$ is strongly convex and twice continuously differentiable on $\mathbb{R}^n$.
\begin{prop}
Let Assumptions (H1) and (H2) hold.
Assume that $f$ is strongly convex on $\mathbb{R}^n$
with constant of strong convexity $\alpha$
with respect to
$\|\cdot\|_2$ and twice continuously differentiable
on $\mathbb{R}^n$.
  Then the dual function $\theta$
given by \eqref{dualfunctionfirst}
 is strongly concave on $\mathbb{R}^q$
 with constant of strong concavity 
not larger than
$\frac{1}{\alpha } \lambda_{\min}(A A^T)$ 
 with respect to norm $\|\cdot\|_2$.
\end{prop}
\begin{proof} By the Implicit Function Theorem, $\theta$ is twice continuously differentiable
with 
$$
\nabla^2 \theta(\lambda)=-H_{x \lambda}^T H_{x x}^{-1} H_{x \lambda}
$$
where 
$$
H_{x \lambda}=\nabla_{x \lambda}^2 
\mathcal{L}(x(\lambda),\lambda)=A^T,\;H_{x x}=\nabla_{x x}^2 \mathcal{L}(x(\lambda),\lambda)=\nabla^2 f(x(\lambda)).
$$
Hence,
$$
\nabla^2 \theta(\lambda)=-A[\nabla^2 f(x(\lambda))]^{-1} A^{T}.
$$
The function $f$ being strongly convex with constant of strong convexity
$\alpha$ we have that $\nabla^2 f(x) \succeq \alpha I_n$ for all $x$
and therefore $\frac{1}{\alpha}I_n \succeq [\nabla^2 f(x)]^{-1}$.
Using Assumption (H2), matrix  $A[\nabla^2 f(x(\lambda))]^{-1} A^{T}$
is invertible for all $\lambda$ and satisfies
$\frac{1}{\alpha} A A^T \succeq A[\nabla^2 f(x(\lambda))]^{-1} A^{T}$
implying $\lambda_{\min}(A[\nabla^2 f(x(\lambda))]^{-1} A^{T}) \leq \frac{1}{\alpha } \lambda_{\min}(A A^T)$
which implies that $\theta$ is strongly concave with constant
of strong concavity not larger than $\frac{1}{\alpha } \lambda_{\min}(A A^T)$ with respect to
norm $\|\cdot\|_2$.\hfill
\end{proof}

\subsection{Applications}

We illustrate Proposition \ref{dualfunctionpbinit0} with 3 examples. The first example is a 
{\em{degenerate}} one and
corresponds to linear programs which indeed satisfy (H1)
and can satisfy (H2). However, as discussed
in Example \ref{lp1} below, for such problems
the domain $\mbox{dom}(\theta)=\{\lambda : \theta(\lambda)>-\infty\}$
of the dual function $\theta$
is either a singleton
or the empty set and such functions are indeed, by
definition, strongly concave even if this property
will not, in this case, be enlightening in practice.
\begin{ex}[Linear programs]\label{lp1} Let $f:\mathbb{R}^n \rightarrow \mathbb{R}$ be given by
\begin{equation}
f(x)=c^T x + c_0 
\end{equation}
where $c \in \mathbb{R}^n$, $c_0 \in \mathbb{R}$. Clearly $f$ is convex differentiable with Lipschitz continuous gradients; any
$L(f)>0$ being a valid Lipschitz constant. Proposition \ref{dualfunctionpbinit0}
tells us that if the rows of $A$ are 
linearly
independent then the dual function $\theta$ of \eqref{pbinit0}
given by \eqref{dualfunctionfirst} is strongly concave on $\mathbb{R}^q$.
In this case, the strong concavity can be checked directly by computing $\theta$.
Indeed, we have 
$$
f^*(x)=\left\{
\begin{array}{l}
-c_0  \;\mbox{    if }x=c,\\
+\infty  \mbox{ if }x \neq c,
\end{array}
\right.
$$
and plugging this expression of $f^*$ into \eqref{thetaintermsoffstar}, we get\footnote{In this simple case, the dual function
is well known and can also be obtained without using the conjugate of $f$}
$$
\theta(\lambda ) =\left\{
\begin{array}{ll}
-\lambda^T b +c_0 & \mbox{if }A^T \lambda = -c,\\
-\infty & \mbox{if }A^T \lambda \neq -c.
\end{array}
\right.
$$
Therefore if $c \in \mbox{Im}(A^T)$ then there is $\lambda \in \mathbb{R}^q$ such that 
\begin{equation}\label{equac}
 A^T \lambda = - c,
\end{equation}
and if the rows of
$A$ are linearly independent then there is only one $\lambda$, let us call it $\lambda_0$, satisfying \eqref{equac}. In this situation,
the domain of $\theta$ is a singleton: $\mbox{dom}(\theta)=\{\lambda_0\}$, and $\theta$ indeed is strongly concave (see Definition \ref{defsconv}).
If $c \notin \mbox{Im}(A^T )$ then $\mbox{dom}(\theta)=\emptyset$ and $\theta$ is again strongly concave. 
\end{ex}
The example which follows gives a class of problems where the dual function is strongly concave on $\mathbb{R}^q$ with $f$ not 
necessarily strongly convex.
\begin{ex}[Quadratic convex programs]\label{exquadprogconcth} Consider a problem of form \eqref{pbinit0}  
where $f(x)=\frac{1}{2} x^T Q_0 x + a_0^T x + b_0$, 
$Q_0$ is an $n\times n$ nonnull semidefinite positive matrix,
$A$ is a $q \times n$ real matrix, $a_0 \in \mbox{Im}(Q_0)$, and $b_0 \in \mathbb{R}$.
Clearly, $f$ is convex, differentiable, and  $\nabla f$ is Lipschitz continuous with Lipschitz constant
$L(f)=\|Q_0\|_2= \lambda_{\max}(Q_0)>0$ with respect to $\|\cdot\|_2$ on $\mathbb{R}^n$.
If the rows of $A$ are linearly independent, using Proposition \ref{dualfunctionpbinit0}
we obtain that the dual function of \eqref{pbinit0} is strongly concave with constant of
strong concavity $\frac{\lambda_{\min}(A A^T)}{\lambda_{\max}(Q_0)}>0$ with respect to norm $\|\cdot\|_2$ on $\mathbb{R}^q$.
Observe that strong concavity holds in particular if $Q_0$ is not definite positive, in which case $f$ is not strongly convex. For this example, strong concavity of $\theta$ is driven by the greatest
eigenvalue of $Q_0$ and by the lowest eigenvalue of
$A A^T$.

Since $f$ is convex, differentiable, its gradient being Lipschitz continuous with Lipschitz constant
$\lambda_{\max}(Q_0)$, from Proposition \ref{conjstconvex}, we know that $f^*$ is strongly convex with constant of
strong convexity $1/\lambda_{\max}(Q_0)$.
This can be checked by direct computation. Indeed, let
$\lambda_{\max}(Q_0)=\lambda_1(Q_0) \geq \lambda_{2}(Q_0) \geq \ldots \geq \lambda_{r}(Q_0)> \lambda_{r+1}(Q_0) = \lambda_{r+2}(Q_0)=\ldots = \lambda_{n}(Q_0)=0$ be the ordered eigenvalues of $Q_0$ where $r$ is the rank of $Q_0$. 
Let $P$ be a corresponding orthogonal matrix of eigenvectors for $Q_0$, i.e., $\emph{Diag}(\lambda_1(Q_0),\ldots,\lambda_n(Q_0))=P^T Q_0 P$ with $P P^T = P^T P =I_n$.
Defining 
$$
Q_0^{+}=P \emph{Diag}\Big(\frac{1}{\lambda_1(Q_0)},\ldots,\frac{1}{\lambda_{r}(Q_0)},\underbrace{0,\ldots,0}_{\mbox{n-r times}}\Big)P^T,
$$
it is straightforward to check that
\begin{equation}\label{expressionfstar}
f^*(x)=
\left\{
\begin{array}{ll}
-b_0+\frac{1}{2}(x-a_0)^T Q_0^{+}(x-a_0)&\mbox{if }x \in \mbox{Im}(Q_0),\\
+ \infty & \mbox{otherwise},
\end{array}
\right.
\end{equation}
and plugging expression \eqref{expressionfstar} of $f^*$ into \eqref{thetaintermsoffstar}, we get
$$
\theta(\lambda)=
\left\{
\begin{array}{ll}
b_0-\lambda^T b -\frac{1}{2}(a_0 + A^T \lambda)^T Q_0^{+}  (a_0 + A^T \lambda) & \mbox{if }A^T \lambda \in \mbox{Im}(Q_0),\\
-\infty & \mbox{otherwise.}
\end{array}
\right.
$$
If $x'=(x'_1,\ldots,x'_n)$ is the vector of the coordinates of $x$ 
in the basis $(v_1,v_2,\ldots,v_n)$ where $v_i$ is $i$th column of $P=[v_1,v_2,\ldots,v_n]$ (i.e., $(v_1,\ldots,v_r)$ is a basis
of Im($Q_0$) and $(v_{r+1},\ldots,v_n)$ is a basis of Ker($Q_0$)) and writing
$a_0 = \sum_{i=1}^r a'_{0 i} v_i$, we obtain
$$
f^*(x)=
\left\{
\begin{array}{l}
g(P^T x) \mbox{ where } g:\mathbb{R}^n \rightarrow \mathbb{R} \mbox{ is given by }g(x')  = -b_0 + \sum_{i=1}^r \frac{(x'_i - a'_{0 i} )^2}{2 \lambda_i( Q_0 )}
 \mbox{ if }x \in \mbox{Im}(Q_0),\\
+\infty \mbox{ otherwise.}
\end{array}
\right.
$$
Observe that for $x', y' \in \mathbb{R}^r \small{\times} \{ \underbrace{(0,\ldots,0)}_{\mbox{n-r times}} \}$
we have
$$
g(y') \geq g(x') + \nabla g(x')^T (y'-x') + \frac{1}{2 \lambda_1(Q_0)}\|y'-x'\|_2^2
$$
and $g$ is strongly convex with constant of strong convexity $\frac{1}{\lambda_1(Q_0)}$
with respect to norm $\|\cdot\|_2$ on $\mathbb{R}^r \small{\times} \{ \underbrace{(0,\ldots,0)}_{\mbox{n-r times}} \}$.
Recalling that $f^*(x)=g(P^T x)$ for $x \in \mbox{dom}(f^*)=\mbox{Im}(Q_0)$ and that
$P^T x \in \mathbb{R}^r \small{\times} \{ \underbrace{(0,\ldots,0)}_{\mbox{n-r times}} \}$ for $x \in \mbox{Im}(Q_0)$, we deduce that $f^*$
is strongly convex with constant of strong convexity
$$
\frac{\lambda_{\min}(P P^T)}{\lambda_1(Q_0)} = \frac{\lambda_{\min}(I_n)}{\lambda_{\max}( Q_0 )} = \frac{1}{\lambda_{\max}( Q_0 )}
$$
with respect to norm $\|\cdot\|_2$.
\end{ex}

\begin{ex}\label{genexample} Let $f(x)=\sum_{k=1}^M \alpha_k f_k(x)$ for $\alpha_k \in \mathbb{R}$
and $f_k: \mathbb{R}^n \rightarrow \mathbb{R}$ convex differentiable with Lipschitz constant
$L_k>0$ with respect to norm $\|\cdot\|_2$ on $\mathbb{R}^n$ for $k=1,\ldots,M$.
Let $A$ be a $q\times n$ matrix with
independent rows. Then the dual function \eqref{dualfunctionfirst}
of \eqref{pbinit0} is strongly concave on $\mathbb{R}^q$ with constant of strong concavity
$\lambda_{\min}(A A^T)/\sum_{k=1}^M \alpha_k L_k$  with respect to $\|\cdot\|_2$.
\end{ex}

\section{Problems with linear and nonlinear constraints}

We now consider problems
of form \eqref{defpblinnonlinstrongconc} with corresponding
dual function
\begin{equation}\label{defthetadualgen1}
\theta(\lambda , \mu ) = 
\left\{
\begin{array}{l}
\inf \;f(x) + \lambda^T ( Ax - b )  + \mu^T g(x) \\
x \in \mathbb{R}^n
\end{array}
\right.
\end{equation}
where 
$g(x)=(g_1(x),\ldots,g_p(x))$.
For this class of problems, the local
strong concavity of dual function \eqref{defthetadualgen1}
is given by the following theorem, which was shown in \cite{vguiguesMP2020}.
\begin{thm}\label{thmstrongconcloc} Consider the optimization  problem 
\begin{equation}\label{defpblinnonlinstrongconc1}
\inf_{x \in \mathbb{R}^n} \{ f(x) :  Ax \leq b, g_i(x) \leq 0, i=1,\ldots,p\},
\end{equation}
where $A$ is a $q\times n$ real matrix.
We assume that
\begin{itemize}
\item[(A1)] $f:\mathbb{R}^n \rightarrow \mathbb{R}$ is 
strongly convex and has Lipschitz continuous gradient on $\mathbb{R}^n$;
\item[(A2)] $g_i:\mathbb{R}^n \rightarrow \mathbb{R}, i=1,\ldots,p$, are convex and have Lipschitz continuous gradients;
\item[(A3)] if $x_*$ is the optimal solution of \eqref{defpblinnonlinstrongconc1} then
the rows of matrix  
 $\left(\begin{array}{c}A \\ J_g(x_* )\end{array}  \right)$ are linearly independent
 where $J_g(x)$ denotes the Jacobian matrix of $g(x)=(g_1(x),\ldots,g_p(x))$ at $x$;
\item[(A4)] there is $x_0 \in \mbox{ri}(\{g \leq 0\})$ such that $A x_0 \leq b$. 
\end{itemize}
Let $\theta$ be the dual function of problem \eqref{defpblinnonlinstrongconc1}:
\begin{equation}\label{defthetadualgen}
\theta(\lambda , \mu ) = 
\left\{
\begin{array}{l}
\inf \;f(x) + \lambda^T ( Ax - b )  + \mu^T g(x) \\
x \in \mathbb{R}^n.
\end{array}
\right.
\end{equation}
Let $(\lambda_* , \mu_* ) \geq 0$ be an optimal solution of the dual problem\footnote{Observe that 
the primal problem has a finite optimal
value and the assumptions of the
Convex Duality Theorem are satisfied
(Slater Assumption (A4) is satisfied
and
the primal objective is bounded from
below on the feasible set), implying
that the dual problem is feasible, has an optimal
solution, and the same optimal
value as the primal problem.}
$$
\sup_{\lambda \geq 0, \mu \geq 0} \theta(\lambda , \mu ).
$$
Then there is some  neighborhood $\mathcal{N}$ of $(\lambda_* , \mu_* )$ such that $\theta$ is strongly concave on $\mathcal{N} \cap \mathbb{R}_{+}^{p+q}$.
\end{thm}

Comparing Theorem \ref{thmstrongconcloc} where strong convexity of the
objective is required with Proposition \ref{dualfunctionpbinit0} which applies to 
problems with
convex (and possibly
non strongly convex) objectives, we can wonder if 
strong convexity can be relaxed to convexity in Theorem \ref{thmstrongconcloc}.
The answer is negative, as shown by the following example.
\begin{ex} Consider the optimization problem 
$$
(P_1) \;\displaystyle  \min_{x \in \mathbb{R}^n} \{c : x_i^2 \leq 1,i=1,\ldots,n\}
$$
which is of form \eqref{defpblinnonlinstrongconc}
with $f(x)=c$ constant, without linear constraints,
 and with constraint functions $g_i(x)=x_i^2-1,$
 $i=1,\ldots,n$, which satisfy Assumption (A2).
 Any feasible $x_*$ with all components nonnull 
 is an optimal solution of $(P_1)$ satisfying Assumption (A3)
 since the rows of $J_g(x_*)=2\emph{Diag}(x_*)$ are linearly independent.
 Clearly (A4) is also satisfied. However, (A1) is not satisfied.
For this example, for $\mu \geq 0$ dual function $\theta$ is given 
by 
$$
\theta(\mu)=\displaystyle c+\min_{x \in \mathbb{R}^n} \sum_{i=1}^n \mu_i(x_i^2-1)
=c-\sum_{i=1}^n \mu_i
$$
and is therefore not strongly concave.  This shows that
the conclusion of Theorem \ref{thmstrongconcloc}
may fail if we replace 
strong convexity by convexity in
Assumption (A1). Observe also that 
nonlinear constraints of $(P_1)$ can be written as
$A x \leq b$ where $b$ is a vector of ones of size $2n$ and 
where the rows of $A=[I_n;-I_n]$ are not linearly independent. 
\end{ex}

It is also natural to wonder if in Proposition \ref{dualfunctionpbinit0}
and Theorem \ref{thmstrongconcloc}, Assumptions (H2) and (A3) can be relaxed
assuming that the gradients of the active constraints at an optimal solution
are linearly independent, instead of assuming that the gradients of all constraints
at an optimal solution are linearly independent. 

For problems with linear constraints
of form \eqref{pbinit0}, from representation 
\eqref{thetaintermsoffstar}, if $\theta$
is strongly concave on $\mathbb{R}^q$
then $\lambda \rightarrow f^*(-A^T \lambda)$
is strongly convex. If $0 \in \mbox{dom}(f^*)$
this implies that Ker$(A^T)=\{0\}$ and therefore that
Assumption (H2) must hold
otherwise $f^*(-A^T \lambda)$
would be constant equal to $f^*(0)$
on the vector space Ker$(A^T)$ of positive dimension
which is not possible for a strongly convex function
with $0 \in \mbox{dom}(f^*)$.
Similarly, the following example shows that 
in Theorem \ref{thmstrongconcloc}, Assumption 
(A3) cannot be relaxed assuming that the gradients
of the active constraints at the optimal solution are linearly
independent. 
\begin{ex} Consider the optimization problem
$$
(P_2) \left\{ 
\begin{array}{l}
\displaystyle \min_{x \in \mathbb{R}^n} \;  \frac{1}{2}\sum_{i=1}^n x_i^2 \\
\displaystyle-\sum_{i=1}^n x_i \leq -1,\\
\displaystyle\sum_{i=1}^n x_i^2 -1 \leq 0,
\end{array}
\right.
$$
of form \eqref{defpblinnonlinstrongconc}
satisfying (A4), with $A=-e^T$ where $e$ is a vector
of ones of dimension $n$,
$f(x)=\frac{1}{2}\sum_{i=1}^n x_i^2$ satisfying (A1),
and $p=1,$ $g_1(x)=\sum_{i=1}^n x_i^2 -1$ satisfying (A2).
The optimal solution of this problem is $x_*=\frac{1}{n}e$
with corresponding optimal value $\frac{1}{2n}$
and only the constraint $-\sum_{i=1}^n x_i \leq -1$
is active at $x_*$. For this problem, for $\lambda, \mu \geq 0$,
dual function $\theta$ is given by
$$
\begin{array}{lcl}
\theta(\lambda,\mu) & = & 
\displaystyle
\min_{x \in \mathbb{R}^n} \frac{1}{2}\sum_{i=1}^n x_i^2 + \lambda (1-\sum_{i=1}^n x_i) + \mu (\sum_{i=1}^n x_i^2 -1)\\
& = &\displaystyle \lambda - \mu - \frac{n}{2} \frac{\lambda^2}{1+2\mu}.
\end{array}
$$
The Hessian matrix of $\theta$ at $(\lambda,\mu)\geq 0$ is given by
$$
\nabla^2 \theta(\lambda,\mu)=
\displaystyle \left( 
\begin{array}{cc}
\displaystyle \frac{-n}{1+2\mu} & \displaystyle \frac{2n \lambda}{(1+2\mu)^2} \\
\displaystyle \frac{2n\lambda}{(1+2\mu)^2} &\displaystyle  \frac{-4n\lambda^2}{(1+2\mu)^3} 
\end{array}
\right).
$$ 
Observe that $0$ is an eigenvalue of $\nabla^2 \theta(\lambda,\mu)$
with $(\frac{2\lambda}{1+2\mu},1)$
a corresponding eigenvector, the other eigenvalue
being $-\frac{n}{1+2\mu}-\frac{4n\lambda^2}{(1+2\mu)^3}$ which is negative
for $\lambda, \mu \geq 0$. Therefore for all $\lambda,\mu \geq 0$
we have that $\nabla^2 \theta(\lambda,\mu)$
is semidefinite negative but not definite negative implying that
$\theta$ is not strongly concave on any set of positive measure
contained in $\mathbb{R}_{+}^2$ and in particular there is no neighborhood
$\mathcal{N}_*$
of the optimal dual solution $\lambda_*, \mu_*$
such that $\theta$ is strongly concave on $\mathcal{N}_* \cap \mathbb{R}_{+}^2$.
Finally, observe that strong duality holds and 
$\lambda_*=\frac{1}{n}$, $\mu_*=0$ since the dual problem is
$$
\begin{array}{lcl}
\displaystyle \max_{\lambda, \mu \geq 0} \theta(\lambda,\mu)
& = &\displaystyle \max_{\lambda \geq 0} \max_{\mu \geq 0} \lambda -\mu -\frac{n}{2}\frac{\lambda^2}{1+2\mu} \\
& = & \max\left(\displaystyle \max_{\lambda \geq \frac{1}{\sqrt{n}}} \frac{1}{2} + \lambda(1-\sqrt{n}),\displaystyle \max_{0 \leq \lambda \leq \frac{1}{\sqrt{n}}} \lambda - \frac{n}{2}\lambda^2\right) \\
& = &
\max\left(\displaystyle -\frac{1}{2} + \frac{1}{\sqrt{n}},\displaystyle \frac{1}{2n}\right)=\displaystyle \frac{1}{2n}, 
\end{array}
$$
whose optimal value is indeed the optimal value $\frac{1}{2n}$
of the primal problem attained at $\lambda_*=\frac{1}{n}$, $\mu_*=0$.
Therefore for this problem, the gradient of the active constraint
at $x_*$ is $-e$ and is consequently linearly independent whereas
the gradients of the constraints at $x_*$ are
$-e$ and $\frac{2}{n}e$ and are therefore not linearly independent.
This shows that the conclusion of 
Theorem \ref{thmstrongconcloc}
does not hold if instead
of assuming that the gradients of all constraint functions 
at the optimal solution $x_*$
are linearly independent
 we assume that the gradients of the active
constraint functions 
at the optimal solution 
are linearly independent.
\end{ex}

\section{Conclusion}

In this paper we analyzed the strong concavity of the dual
function of an optimization problem. A possible extension
would be to show this property for some classes of problems
when the dual function is obtained dualizing only some of the
constraints.

\section*{Acknowledgments} Research of the author was supported by 
CNPq grant 304887/2019-6. 
The author would like to thank 
the referee and
the Managing Editor in charge of the manuscript  
for helpful suggestions
that improved the presentation.

\addcontentsline{toc}{section}{References}
\bibliographystyle{plain}
\bibliography{Strong_Concavity_Dual}

\end{document}